\documentclass{amsart}

\usepackage{amsmath}
\usepackage{textcomp}
\usepackage{amsfonts}
\usepackage{amssymb,enumerate}
\usepackage{amsthm}
\usepackage{stmaryrd}
\usepackage[all]{xy}
\usepackage{hyperref}

\newcommand{\Jac}{\operatorname{Jac}}
\newcommand{\Reg}{\operatorname{Reg}}
\newcommand{\Spec}{\operatorname{Spec}}
\newcommand{\Supp}{\operatorname{Supp}}
\newcommand{\V}{\operatorname{V}}
\newcommand{\Max}{\operatorname{Max}}
\newcommand{\Z}{\operatorname{Z}}

\newcommand{\fm}{\frak{m}}
\newcommand{\fp}{\frak{p}}

\newcommand{\fq}{\frak{q}}
\newcommand{\fa}{\frak{a}}
\newcommand{\fb}{\frak{b}}

\newtheorem{thm}{Theorem}[section]
\newtheorem{cor}[thm]{Corollary}
\newtheorem{lem}[thm]{Lemma}
\newtheorem{prop}[thm]{Proposition}

\newtheorem{exam}[thm]{Example}
\newtheorem{rem}[thm]{Remark}

\begin{document}

\bibliographystyle{amsplain}

\date{}

\author{Y. Azimi, P. Sahandi and N. Shirmohammadi}

\address{Department of Pure Mathematics, Faculty of Mathematical Sciences, University of Tabriz, Tabriz,
Iran.} \email{u.azimi@tabrizu.ac.ir}

\address{Department of Pure Mathematics, Faculty of Mathematical Sciences, University of Tabriz, Tabriz,
Iran.} \email{sahandi@ipm.ir}

\address{Department of Pure Mathematics, Faculty of Mathematical Sciences, University of Tabriz, Tabriz,
Iran.} \email{shirmohammadi@tabrizu.ac.ir}




\title[Pr\"ufer conditions under the amalgamated construction]{Pr\"ufer conditions under the amalgamated construction}

\begin{abstract}
In this paper we improve the recent results on the transfer of
Pr\"ufer, Gaussian and arithmetical conditions on amalgamated
constructions. As an application we provide an answer to a question
posed by Chhiti, Jarrar, Kabbaj and Mahdou as well as we construct
various examples.
\end{abstract}

\maketitle

\section{Introduction}

The notion of Pr\"{u}fer domain appeared for the first time in
\cite{p}. Then Krull further studied and named such domains as
Pr\"{u}fer domains \cite{k36}. Pr\"ufer domains play a central role
in \emph{Multiplicative ideal theory} and several different
characterizations of these domains exist \cite{g72}. Many of these
characterizations have been extended to the case of rings with
zero-divisors. Among them are the Pr\"{u}fer, Gaussian and
arithmetical rings. It is commonly accepted to define
\emph{Pr\"{u}fer rings} as the rings in which every non-zero
finitely generated regular ideal is projective \cite{bs67, g69}. A
commutative ring $R$ is called \emph{Gaussian} if
$c_R(fg)=c_R(f)c_R(g)$ for every two polynomials $f,g\in R[X]$,
where $c_R(f)$ is the content ideal of $f$ \cite{T}. Finally, a
commutative ring is called an \emph{arithmetical} ring if every
finitely generated ideal is locally principal \cite{j66}. It is well
known that an arithmetical ring is Gaussian and a Gaussian ring is
Pr\"{u}fer and that all these coincide in the context of domains.
The aim of this paper is to investigate when the amalgamated algebra
is a Pr\"{u}fer, Gaussian or arithmetical ring.

D'Anna, Finocchiaro, and Fontana in \cite{DFF} and \cite{DFF2} have
announced the following ring construction. Let $R$ and $S$ be two
commutative rings with unity, let $J$ be an ideal of $S$ and let
$f:R\to S$ be a ring homomorphism. They introduced the following
subring
$$R\bowtie^fJ:=\{(r,f(r)+j)\mid r\in R\text{ and }j\in J\}$$ of
$R\times S$, called the \emph{amalgamation of $R$ with $S$ along $J$
with respect to $f$}. This construction generalizes the amalgamated
duplication of a ring along an ideal (introduced and studied in
\cite{DF}). Moreover, several classical constructions such as the
Nagata's idealization, the $R + XS[X]$ and
the $R+XS\llbracket X \rrbracket$ constructions can be studied as
particular cases of this new construction (see \cite[Example 2.5 and Remark 2.8]{DFF}).

Assuming that $J$ and $f^{-1}(J)$ are regular ideals, Finocchiaro in
\cite[Theorem 3.1]{f} obtained that $R\bowtie^fJ$ can be Pr\"ufer
ring only in the trivial case, namely when $R$ and $S$ are Pr\"ufer
rings and $J=S$. He also established some results about the transfer
of some Pr\"ufer-like conditions between $R\bowtie^fJ$ and $R$.
Meanwhile, among other things, the authors in \cite[Theorem 2.2]{k}
find necessary and sufficient conditions under which Pr\"ufer-like
properties transfer between the amalgamated duplication $R\bowtie
I=R\bowtie^{f}I$ of a local ring $R$ and the ring $R$ itself, where
$I$ is an ideal of $R$, $S=R$ and $f$ is the identity map on $R$. In
particular, they proved that $R\bowtie I$ is a Pr\"ufer ring if and
only if $R$ is a Pr\"ufer ring and $I = rI$ for every regular
element $r$ of the unique maximal ideal of $R$. They then asked that
is their characterization valid in the global case? i.e., when $R$
is Pr\"{u}fer (not necessarily local) or locally Pr\"{u}fer
\cite[Question 2.11]{k}.

The outline of the paper is as follows. In Section 2, we determine
the zero-divisors of the amalgamated algebra $R\bowtie^f J$ under
certain conditions. In Section 3, among other things, we attempt to generalize the
results mentioned above to the case of the amalgamated algebra
$R\bowtie^f J$ and, consequently, to answer \cite[Question 2.11]{k}.
In Sections 4 and 5, we investigate the transfer of Gaussian and
arithmetical conditions of the amalgamated algebra $R\bowtie^f J$.
By various examples we examine our results.

\section{Preliminaries}

To determine whether a ring is Pr\"{u}fer  (resp. total ring of
quotients) or not, it is of great importance to know zero-divisors
and regular elements of the ring. This section is devoted to
determine the zero-divisor and regular elements of $R\bowtie^f J$.
Let first us fix some notation which we shall use throughout the
paper: $R$ and $S$ are two commutative rings with unity, $J$ a
proper ideal of the ring $S$, and $f:R\to S$ is a ring homomorphism.
For a commutative ring $A$, $\Max(A)$ denotes the set of maximal
ideals of $A$, $\Jac(A)$ denotes the Jacobson radical of $A$,
$\Reg(A)$ denotes the set of regular elements of $A$, $\Z(M)$
denotes the set of zero-divisors of an $A$-module $M$ and, for an
ideal $I$ of $A$, $\V(I)$ denotes the set of all prime ideals of $A$
containing $I$. An ideal of ring is called a \emph{regular ideal} if
it contains a regular element.

A full description for the set of zero-divisors of the duplication
ring $R\bowtie I$ has been provided in \cite[Proposition 2.2]{y}. In
the following lemma we generalize it to amalgamated algebra. Before
we announce this generalization, we should perhaps point out that
$\{(r,f(r)+j)\mid j'(f(r)+j)=0$, for some $ j'\in J\setminus \{0\}
\}\subseteq \Z(R\bowtie^f J)$ in general.

\begin{lem}\label{zdd}
There is the inclusion $\Z(R\bowtie^f J)\subseteq \{(r,f(r)+j)\mid
r\in \Z(R)\}\cup \{(r,f(r)+j)\mid j'(f(r)+j)=0$, for some $j'\in
J\setminus \{0\} \}$, with equality if at least one of the following
conditions hold:
\begin{itemize}
 \item [(1)] $f(\Z(R))\subseteq J$ and $f^{-1}(J)\neq0$;
 \item [(2)] $f(\Z(R))J=0$ and $f^{-1}(J)\neq0$;
 \item [(3)] $J\subseteq f(R)$;
 \item [(4)] $J$ be a torsion $R$-module.
\end{itemize}
\end{lem}
\begin{proof}
Assume that $(r,f(r)+j)\in\Z(R\bowtie^f J)$. Then
$(r,f(r)+j)(s,f(s)+j')=0$ for some non-zero element $(s,f(s)+j')\in
R\bowtie^f J$. Hence $rs=0$ and $jf(s)+j'(f(r)+j)=0$. If $s=0$, then
$j'\neq0$ and $j'(f(r)+j)=0$. Otherwise, $r\in \Z(R)$. This proves
the inclusion.

We already have the inclusion $\{(r,f(r)+j)\mid j'(f(r)+j)=0$, for
some $j'\in J\setminus \{0\} \}\subseteq \Z(R\bowtie^f J)$. To
complete the proof, it is enough for us to show that
$\{(r,f(r)+j)\mid r\in\Z(R)$ and $j\in J\}\subseteq\Z(R\bowtie^f J)$
under the validity of any of conditions (1)--(4). Let $(r,f(r)+j)\in
R\bowtie^f J$ where $r\in \Z(R)$ and $j\in J$. Then there is a
non-zero element $s\in R$ such that $rs=0$.

(1) If $r\neq0$, we have $(r,f(r)+j)(s,f(s)+(-f(s)))=0$. So assume
that $r=0$ and choose $0\neq a\in f^{-1}(J)$. Then $(0,i)(a,0)=0$.

(2) If $r\neq0$, we have $(r,f(r)+j)(s,f(s))=0$. So assume that
$r=0$ and choose $0\neq a\in f^{-1}(J)$. Then $(0,i)(a,0)=0$.

(3) If $f(s)J=0$, then $(r,f(r)+j)(s,f(s))=0$. Assume now that
$f(s)J\neq0$. Then there is $0\neq k\in J$ such that $f(s)k\neq 0$.
By assumption $k=f(t)$ for some $t\in R$. It is clear that $st\neq0$
and $(r,f(r)+i)(st,f(st)-f(s)k)=0$.

(4) Let $t\in \Reg(R)$ such that $f(t)j=0$. It is easy to see that
$ts\neq 0$ and $(r,f(r)+i)(ts,f(ts))=0$.
\end{proof}

We say that the amalgamated ring $R\bowtie^f J$ has the
\emph{condition $\star$} if the equality
$$\Z(R\bowtie^f J)=\{(r,f(r)+j)\mid r\in \Z(R)\}\cup \{(r,f(r)+j)\mid
j'(f(r)+j)=0, \exists j'\in J\setminus \{0\} \}$$ holds. Lemma
\ref{zdd} now says that the amalgamated ring $R\bowtie^f J$
involving one of the mentioned conditions has the condition $\star$.
In particular when $f$ is surjective, by part (3) of Lemma
\ref{zdd}, $R\bowtie^f J$ has the condition $\star$. Therefore the
amalgamated duplication of a ring along an ideal has the condition
$\star$. Although the trivial extension of a ring by a module does
not satisfy any conditions mentioned in Lemma \ref{zdd}, the next
remark illustrates that the trivial extension of a ring by a module
also has this condition. So we have a bunch of examples satisfying
the condition $\star$.

Let $M$ be an $R$-module. Nagata (1955) gave a ring
extension of $R$ called the \emph{trivial extension} of $R$ by $M$
(or the \emph{idealization} of $M$ in $R$), denoted here by
$R\ltimes M$ \cite[page 2]{Na}. It should be noted that the module
$M$ becomes an ideal in $R\ltimes M$ and $(0\ltimes M)^2=0$. As in
\cite[Remark 2.8]{DFF}, if $S:=R\ltimes M$, $J:=0\ltimes M$, and
$\iota:R\to S$ be the natural embedding, then $R\bowtie^\iota J\cong
R\ltimes M$ which maps the element $(r,\iota(r)+(0,m))$ to the
element $(r,m)$.

\begin{rem}\label{Axtel}
Let $M$ be an $R$-module. By \cite[Proposition 1.1]{as06}, one has
$$\Z(R\ltimes M)=\{(r,m)\mid
r\in\Z(R)\cup\Z(M)\text{ and }m\in M\}.$$
This equality together
with the isomorphism mentioned above shows that the trivial
extension $R\ltimes M$ has the condition $\star$.
\end{rem}

\section{Transfer of Pr\"{u}fer condition}

In this section we investigate the transfer of Pr\"{u}fer condition
on the amalgamated algebra $R\bowtie^f J$. The following concept and
lemmas are crucial in this investigation.

Let $\fp$ be a prime ideal of $R$. Then $(R, \fp)$ is said to has
the \emph{regular total order property} if, for each pair of ideals
$\fa$ and $\fb$ of $R$, at least one of which is regular, the ideals
$\fa R_\fp$ and $\fb R_\fp$ are comparable. Using this notion,
Griffin \cite[Theorem 13]{g69} obtained the following useful
characterization of Pr\"{u}fer rings which we use it frequently.

\begin{thm} \label{rto}
The ring $R$ is Pr\"{u}fer if and only if $(R,\fm)$ has the regular
total order property for every $\fm\in\Max(R)$.
\end{thm}

The following basic lemma will be used throughout this section.

\begin{lem}\label{zd}
Let $U$ be a multiplicatively closed subset of the ring $R$. If
$r\in\Reg(R)$, then $r/1\in\Reg(R_U)$. The converse holds whenever
$U\cap \Z (R)=\phi$. In particular, if $\Z(R)\subseteq \Jac (R)$,
then, for a maximal ideal $\fm$, one has $r/1\in \Reg(R_\fm)$ if and
only if $r\in \Reg (R)$.
\end{lem}

This lemma enables us to present the following version
of the theorem above which is very useful for us.

\begin{lem}\label{rto p}
Assume that $\Z(R)\subseteq \Jac (R)$ and that $\fm$ is a maximal
ideal of the ring $R$. Then $(R,\fm)$ has the regular total order
property if and only if the principal ideals $xR_\fm$ and $yR_\fm$
are comparable for all $x,y\in R$ with $x$ is regular.
\end{lem}
\begin{proof}
Assume that the principal ideals $xR_\fm$ and $yR_\fm$ are
comparable for all $x,y\in R$ with $x$ is regular. Let $\fa$ and
$\fb$ be ideals of $R$, and $x$ be a regular element of $R$ in
$\fa$. Suppose that $y/1\in \fb_\fm \setminus \fa_\fm$. It follows
that $xR_\fm \subseteq yR_\fm$. To conclude the inclusion
$\fa_{\fm}\subseteq \fb_\fm$, it is enough for us to show that
$\fa_\fm \subseteq yR_\fm$. Suppose on the contrary that there
exists an element $a/1\in \fa_\fm \setminus yR_\fm$. Since $aR_\fm
\nsubseteq xR_\fm$, it follows that $xR_\fm \subseteq aR_\fm$. It
turns out that $a/1$ is regular. This in conjunction with Lemma
\ref{zd} implies that $a$ is regular. Hence $yR_\fm$ and $aR_\fm$
are comparable which is a contradiction.
\end{proof}

In the sequel, we will use the following
properties of amalgamated rings without explicit comments.

\begin{rem}\label{rem}
The following statements hold.
\begin{enumerate}
\item (\cite[Corollaries 2.5 and 2.7]{DFF1})
For $\fp\in\Spec(R)$ and $\fq\in\Spec(S)\setminus \V(J)$, set
\begin{align*}
\fp^{\prime_f}:= & \fp\bowtie^fJ:=\{(p,f(p)+j)|p\in \fp, j\in J\}, \\[1ex]
\overline{\fq}^f:= & \{(r,f(r)+j)|r\in R, j\in J, f(r)+j\in \fq\}.
\end{align*}
Then, one has the following.
\begin{enumerate}
\item The prime ideals of $R\bowtie^fJ$ are of the type $\overline{\fq}^f$ or $\fp^{\prime_f}$, for
$\fq$ varying in $\Spec(S)\setminus \V(J)$ and $\fp$ in $\Spec(R)$.
\item $\Max(R\bowtie^fJ)=\{\fp^{\prime_f}|\fp\in\Max(R)\}\cup\{\overline{\fq}^f|\fq\in\Max(S)\setminus \V(J)\}$.
\end{enumerate}
\item (\cite[Proposition 2.9]{DFF1}) The following formulas for localizations hold.
\begin{enumerate}
\item For any $\fq\in\Spec(S)\setminus\V(J)$, the localization $(R\bowtie^fJ)_{\overline{\fq}^f}$ is
canonically isomorphic to $S_{\fq}$. This isomorphism maps the
element $(r,f(r)+j)/(r',f(r')+j')$ to $(f(r)+j)/(f(r')+j')$.
\item For any $\fp\in\Spec(R)\setminus\V(f^{-1}(J))$, the
localization $(R\bowtie^fJ)_{\fp^{\prime_f}}$ is canonically
isomorphic to $R_{\fp}$. This isomorphism maps the element
$(r,f(r)+j)/(r',f(r')+j')$ to $r/r'$.
\item For any $\fp\in\Spec(R)$ containing $f^{-1}(J)$, consider the multiplicative
subset $T_{\fp}:=f(R\setminus\fp)+J$ of $S$ and set
$S_{T_{\fp}}:=T_{\fp}^{-1}S$ and $J_{T_{\fp}}:=T_{\fp}^{-1}J$. If
$f_{\fp}:R_{\fp}\to S_{T_{\fp}}$ is the ring homomorphism induced by
$f$, then the ring $(R\bowtie^fJ)_{\fp^{\prime_f}}$ is canonically
isomorphic to $R_{\fp}\bowtie^{f_{\fp}}J_{T_{\fp}}$. This
isomorphism maps the element $(r,f(r)+j)/(r',f(r')+j')$ to
$(r/r',(f(r)+j)/(f(r')+j'))$.
\end{enumerate}
\end{enumerate}
\end{rem}

We are now ready to state and prove the main result of this section.
In fact, it provides a partial converse of \cite[Proposition 4.2]{f} as
well as a generalization of \cite[Theorem 2.2]{k}.

\begin{thm}\label{amalg}
Assume that $f(\Reg (R))\subseteq \Reg (S)$.
\begin{enumerate}
\item If $R\bowtie^f J$ is a Pr\"{u}fer ring, then $R$ is a Pr\"{u}fer ring
and $J_{T_\fm}=f(r) J_{T_\fm}$ for every $\fm\in\Max(R)$ and every
$r\in\Reg(R)$.
\item Assume that $R\bowtie^f J$ has the condition $\star$ and that $\Z(R\bowtie^f J)\subseteq\Jac(R\bowtie^f J)$. If  $R$ is a Pr\"{u}fer
ring and $J_{T_\fm}=f(r) J_{T_\fm}$ for every $\fm\in\Max(R)$ and
every $r\in\Reg(R)$, then $R\bowtie^f J$ is a Pr\"{u}fer ring.
\end{enumerate}
\end{thm}
\begin{proof}
(1) Assume that $R\bowtie^f J$ is a Pr\"ufer ring. By
\cite[Proposition 4.2]{f}, $R$ is a Pr\"ufer ring. Let
$r\in\Reg(R)$, $\fm\in\Max(R)$, and $k/x\in J_{T_\fm}$. It is easy
to see that $(r,f(r))\in\Reg(R\bowtie^f J)$. By Theorem \ref{rto},
the ideal $(r,f(r))(R\bowtie^f J)_{\fm^{'f}}$ is comparable with any
ideal of $(R\bowtie^f J)_{\fm^{'f}}$. Hence the principal ideal of
$R_\fm \bowtie^{f_\fm} J_{T_\fm}$ generated by $(r/1,f_\fm(r/1))$ is
comparable with the principal ideal generated by $(0,k/x)$. If
$(r/1,f_\fm(r/1))\left(R_\fm \bowtie^{f_\fm}
J_{T_\fm}\right)\subseteq (0,k/x)\left(R_\fm \bowtie^{f_\fm}
J_{T_\fm}\right)$, then $(r/1,f_\fm (r/1))=(0,k/x)
(s/z,f_\fm(s/z)+j/y)$ for some $(s/z,f_\fm(s/z)+j/y) \in R_\fm
\bowtie^{f_\fm} J_{T_\fm}$ which implies that $r/1=0$. This is a
contradiction since $r/1\in\Reg(R_\fm)$ by Lemma \ref{zd}. So we may
assume that the inclusion $(0,k/x)\left(R_\fm \bowtie^{f_\fm}
J_{T_\fm}\right)\subseteq(r/1,f_\fm(r/1))\left(R_\fm \bowtie^{f_\fm}
J_{T_\fm}\right)$ holds. Thus there exists $(s/z,f_\fm(s/z)+j/y)\in
R_\fm \bowtie^{f_\fm}J_{T_\fm}$ such that
$(0,k/x)=(r/1,f_\fm(r/1))(s/z,f_\fm(s/z)+j/y)$. This means that
$s/z=0$ and $k/x=f_\fm( r/1)j/y=f(r)j/y\in f(r) J_{T_\fm}$, as
desired.

(2) By Theorem \ref{rto}, it is sufficient to prove that
$(R\bowtie^f J, \mathcal{M})$ has the regular total order property
for every maximal ideal $\mathcal{M}$ of $R\bowtie^f J$. On the
other hand, it follows from the hypothesis that $J\subseteq \Jac(S)$
and $\Z(R)\subseteq\Jac(R)$. Indeed if $J\nsubseteq\Jac(S)$, there
exists a $\fq\in\Max(S)\setminus \V(J)$. Then $\{(0,j)\mid j\in
J\}\subseteq \Z(R\bowtie^f J)\subseteq\Jac(R\bowtie^f
J)\subseteq\overline{\fq}^f$. This yields $J\subseteq \fq$ a
contradiction. For the second inclusion assume $a\in\Z(A)$ and
$\fm\in\Max(R)$. Then $(a,f(a))\in\Z(R\bowtie^f
J)\subseteq\Jac(R\bowtie^f J)\subseteq\fm^{\prime_f}$, which implies
that $a\in\fm$. Therefore we have two cases to consider:

\textbf{Case 1.} Assume that $\mathcal M= \fm^{\prime_f}$ for some
$\fm\in\Max(R)$ such that $f^{-1}(J)\nsubseteq\fm$. Let
$(r,f(r)+j)\in \Reg (R\bowtie^f J)$. Then $(r,f(r)+j)/(1,1)\in\Reg
((R\bowtie^f J)_{\fm^{\prime_f}})$. Using the isomorphism
$(R\bowtie^f J)_{\fm^{\prime_f}}\cong R_\fm$, one has
$r/1\in\Reg(R_\fm)$ which implies that $r\in\Reg(R)$ by Lemma
\ref{zd}. Since $(R,\fm)$ has the regular total order property, the
ideal $rR_\fm$ is comparable with every ideal of $R_\fm$.
Consequently, the ideal $(r,f(r)+j)(R\bowtie^f J)_{\fm^{\prime_f}}$
is comparable with every ideal of $(R\bowtie^f J)_{\fm^{\prime_f}}$.

\textbf{Case 2.} Assume that $\mathcal M= \fm^{\prime_f}$ for some
$\fm\in\Max(R)$ such that $f^{-1}(J)\subseteq\fm$. Let
$(r,f(r)+j)\in \Reg (R\bowtie^f J)$ and $(r',f(r')+j')\in R\bowtie^f
J$. Since $R\bowtie^fJ$ has the condition $\star$, we have
$r\in\Reg(R)$. Thus the principal ideals $rR_\fm$ and $r'R_\fm$ are
comparable because $(R,\fm)$ has the regular total order property.
We may and do assume $r'R_\fm\subseteq rR_\fm$. So we have
$r'/1=(r/1)(t/u)$ for some $t/u\in R_\fm$. We claim that there
exists an element $k/y\in J_{T_\fm}$ such that $j'/1-f_\fm
(t/u)j/1=(k/y)(f_\fm (r/1)+j/1)$. If this is the case, then one has
the equality $$(r'/1,f_\fm(r'/1)+j'/1)=(t/u,f_\fm (t/u)
+k/y)(r/1,f_\fm(r/1)+j/1)$$ in the ring $R_\fm \bowtie^{f_\fm}
J_{T_\fm}$ which shows the following inclusion of principal ideals
$$(r'/1,f_\fm(r'/1)+j'/1)(R_\fm \bowtie^{f_\fm} J_{T_\fm})\subseteq(r/1,f_\fm(r/1)+j/1)(R_\fm \bowtie^{f_\fm} J_{T_\fm}).$$
This inclusion yields the inclusion $(r',f(r')+j')(R\bowtie^f
J)_{\fm^{\prime_f}}\subseteq(r,f(r)+j)(R\bowtie^f
J)_{\fm^{\prime_f}}$. It remains to prove the claim. To this end,
one notices that our assumption gives an element $l/v\in J_{T_\fm}$
such that $j/1=f_\fm (r/1)(l/v)$. Hence $f_\fm(r/1)+j/1=f_\fm
(r/1)((v+l)/v)$ and observe that the element $(v+l)/v$ is a unit
element of $S_{T_\fm}$.
\end{proof}

We have the following corollaries.

\begin{cor}
Assume that $(R,\fm)$ is a local ring, $J\subseteq\Jac(S)$ and
$f(\Reg (R))\subseteq \Reg (S)$.
\begin{enumerate}
\item If $R\bowtie^f J$ is a Pr\"{u}fer ring, then $R$ is a Pr\"ufer ring and
$J=f(r)J$ for every $r\in\Reg(R)$.
\item Assume further that $R\bowtie^f J$ has the condition $\star$. If  $R$ is a Pr\"ufer ring
and $J=f(r)J$ for every  $r\in \Reg(R)$, then $R\bowtie^f J$ is a
Pr\"ufer ring.
\end{enumerate}
\end{cor}
\begin{proof}
It follows from the inclusion $J\subseteq\Jac(S)$ that the elements
of $T_{\fm}=f(R\setminus\fm)+J$ are units in $S$ for every $\fm\in
\Max(R)$. Hence, for every $r\in R$ and every $\fm\in \Max(R)$,
$J_{T_{\fm}}=f(r) J_{T_{\fm}}$ holds if and only if $J=f(r)J$ holds.
One also notices from Remark \ref{rem}(1)(b) that $(R,\fm)$ is a
local ring and $J\subseteq\Jac(S)$ if and only if $R\bowtie^f J$ is
local. The result is therefore immediate from Theorem \ref{amalg}.
\end{proof}

We now deal with the amalgamated duplication. Assume that $R=S$ and
$f$ is the identity map on $R$. Let $I$ be an ideal of $R$, $0\neq
r\in R$ and $\fm\in\Max(R)$. One notices that if $\fm\supseteq I$,
then $T_\fm=R\setminus\fm+I=R\setminus\fm$, otherwise, $I_{T_\fm}=0$
by \cite[Remark 2.4]{f}. Hence $I_{T_\fm}=rI_{T_\fm}$ holds if and
only if $I_\fm=rI_\fm$. It also easily obtain that $\Z(R\bowtie
I)\subseteq\Jac(R\bowtie I)$ if and only if $\Z(R)\subseteq\Jac(R)$
and $I\subseteq\Jac(R). $Therefore we can derive the following
corollary of Theorem \ref{amalg}. This corollary generalizes
\cite[Theorem 2.2]{k} to non-local case as well as provides an
answer to \cite[Question 2.11]{k}.

\begin{cor}
Let $I$ be an ideal of the ring $R$.
\begin{enumerate}
\item If $R\bowtie I$ is a Pr\"{u}fer ring, then $R$ is a Pr\"ufer ring and
$I_\fm=rI_\fm$ for every $\fm\in\Max(R)$ and every $r\in\Reg(R)$.
\item Assume further that $\Z(R)\subseteq\Jac(R)$ and $I\subseteq\Jac(R)$. If $R$ is a Pr\"ufer ring
and $I_\fm=rI_\fm$ for every $\fm\in\Max(R)$ and every
$r\in\Reg(R)$, then $R\bowtie I$ is a Pr\"ufer ring.
\end{enumerate}
\end{cor}

\begin{cor} (See \cite[Theorem 2.2]{k})
Let $(R,\fm)$ be a local ring and $I$ a proper ideal of $R$. Then
$R\bowtie I$ is a Pr\"ufer ring if and only if $R$ is a Pr\"ufer
ring and $I=rI$ for every $r\in\Reg(R)$.
\end{cor}

We shall now describe the behaviour of Pr\"{u}fer condition on trivial extension.

\begin{cor}\label{trivial}
Let $M$ be an $R$-module such that $\Z(M)\subseteq\Z(R)$.
\begin{enumerate}
\item If $R\ltimes M$ is a Pr\"{u}fer ring, then $R$ is a Pr\"ufer ring and
        $M_\fm=r M_\fm$ for every $\fm \in \Max (R)$ and every $r\in\Reg(R)$.
\item Assume further that $\Z (R)\subseteq \Jac (R)$. If  $R$ is a Pr\"ufer
ring and $M_\fm=r M_\fm$ for every $\fm\in\Max(R)$ and every $r\in
\Reg(R)$, then $R\ltimes M$ is a Pr\"ufer ring.
\end{enumerate}
\end{cor}
\begin{proof}
It follows from the inclusion $\Z(M)\subseteq\Z(R)$ that
$\iota(\Reg(R))\subseteq\Reg(R\ltimes M)$. On the other hand, for
$r\in R$ and $\fm\in\Max(R)$, we have
$T_\fm=\iota(R\setminus\fm)+J=R\ltimes M\setminus\fm\ltimes M$ where
$J=0\ltimes M$. Hence $J_{T_\fm}=\iota(r) J_{T_\fm}$ holds if and only if
$M_\fm=rM_\fm$ holds. Finally, it is easy to see that our assumption
$\Z(M)\subseteq\Z(R)$ together with Remark \ref{Axtel} implies
$\Z(R\ltimes M)\subseteq \Jac(R\ltimes M)$. Consequently, Theorem \ref{amalg}
completes the proof.
\end{proof}

It is worth pointing out that in the corollary above the assumption
$\Z(M)\subseteq\Z(R)$ in crucial. For example let $(R,\fm)$ be a
local integral domain which is not a valuation domain (e.g.
$R:=k+Yk(X)\llbracket Y \rrbracket$, where $k$ is a field and $X,Y$
are indeterminates over $k$). The trivial extension
$R\ltimes(R/\fm)$ is a total ring of quotients, hence, Pr\"{u}fer,
while $R$ is not (see \cite[Proposition 3.1(a)]{L86}).

\begin{cor}
Assume that $R$ is an integral domain or is a local ring and that
$M$ is a torsion-free $R$-module. Then $R\ltimes M$ is a Pr\"{u}fer
ring if and only if $R$ is a Pr\"ufer ring and $M_\fm=r M_\fm$ for every
$\fm\in\Max(R)$ and $r\in \Reg(R)$.
\end{cor}

Applying the corollary above we obtain an easy proof for the
following result of Bakkari, Kabbaj, and Mahdou.

\begin{cor} {\em (cf. \cite[Theorem 2.1(1)]{BKM})}
Assume that $A\subseteq B$ is an extension of domains and let $K$ be
the quotient field of $A$. Then $A\ltimes B$ is a Pr\"{u}fer ring if
and only if $A$ is a Pr\"ufer domain and $K\subseteq B$.
\end{cor}
\begin{proof}
It is well-known that $B=\cap_{\fm\in\Max(A)}B_{\fm}$. Hence, for
every $0\neq a\in A$, one has $B_\fm=aB_\fm$ for every
$\fm\in\Max(A)$ if and only if $B=aB$. This is the case if and only
if $K\subseteq B$. The conclusion is now clear by Corollary \ref{trivial}.
\end{proof}

{\em Total rings of quotients} are  trivial examples of Pr\"ufer rings.
 In \cite[Proposition 4.6]{f}, Finocchiaro  gave conditions, under which, total
 ring of quotients property transfers from $R$ to $R\bowtie^f J$.
The next proposition provides a slight
generalization as well as is a partial converse of it.

\begin{prop}\label{T}
The following statements hold:
    \begin{enumerate}
        \item Assume that $J\subseteq\Jac(S)$ and that $R\bowtie^f J$
        has the condition $\star$. If $R$ is a total
        ring of quotients, then so is $R\bowtie^f J$.
        \item Assume that $f(\Reg(R))\subseteq\Reg(S)$. If $R\bowtie^f J$ is a total
        ring of quotients, then so is $R$.
    \end{enumerate}
\end{prop}
\begin{proof}
The proof of (1) is similar to that of \cite[Proposition 4.6]{f}, but,
for the reader's convenience, we derive it under the general assumption
that $R\bowtie^f J$ has the condition $\star$. Assume that $R$ is a total
ring of quotients and let $(r,f(r)+j)$ be a non-unit element of $R\bowtie^f J$.
Hence $(r,f(r)+j)\in \fm^{\prime_f}$ for some $\fm\in\Max(R)$; so that
$r\in\fm$. This in conjunction with the assumption that $R$ is a total
ring of quotients yields $r\in\Z(R)$. It now follows from Lemma \ref{zdd}
that $(r,f(r)+j)\in \Z(R\bowtie^f J)$. Therefore $R\bowtie^f J$ is a total
ring of quotients.
To prove (2), let $r\in R$. If $(r,f(r))$ is a unit of
$R\bowtie^f J$, then obviously $r$ is a unit of $R$. Otherwise
$(r,f(r))$ is a zero-divisor of $R\bowtie^f J$. Our desired result
will be established if $r$ is a zero-divisor of $R$. If not, by Lemma
\ref{zdd}, there exists $j\in J$ such that $jf(r)=0$, which
contradicts since $f(r)$ is regular.
\end{proof}

\section{Transfer of Gaussian condition}

Our goal in this section is to determine the behaviour
of Gaussian condition under amalgamated construction.
Let $X$ be an indeterminate over $R$. The \emph{content ideal}
$c_R(f)$ of a polynomial $f\in R[X]$ is defined to be the ideal of
$R$ generated by the coefficients of $f$. A ring $R$ is called a
\emph{Gaussian ring} if $c_R(fg)=c_R(f )c_R(g)$ for every two
polynomials $f,g\in R[x]$. It is shown in \cite[Theorem 2.2]{T} that
a local ring $R$ is Gaussian if and only if, for every two elements
$a,b \in R,$ the following two properties hold: (i) $(a,b)^2= (a^2)$
or $(b^2)$; (ii) if $(a,b)^2= (a^2)$ and $ab=0$, then $b^2=0$.
Using this characterization, the authors in \cite[Theorem 3.2(2) and Corollary 3.8(2)]{k}
established that the amalgamated duplication ring $R\bowtie I$
is Gaussian if and only if $R$ is Gaussian, $I_\fm^2=0$ and $rI_\fm=r^2I_\fm$
for every $\fm\in\Max(R)\cap\V(I)$ and every $r\in\fm$.
Then, again, by means of this characterization,
for a local ring $(R,\fm)$ and assuming
$J^2=0$, it is shown that $R\bowtie^f J$ is Gaussian if and only if
so is $R$ and $f(r)J=f(r)^2J$ for every $r\in\fm$ \cite[Theorem
2.1(2)]{ga}. In the following,
we give an improvement to this result provided that
$J\subseteq f(R)$.

\begin{thm}\label{g}
Assume that $(R,\fm)$ is a local ring and $J\subseteq f(R)\cap \Jac
(S)$ (e.g. $f$ is surjective). Then $R\bowtie^f J$ is Gaussian if
and only if $R$ is Gaussian, $J^2=0$ and $f(r)J=f(r)^2J$ for every
$r\in \fm$.
\end{thm}
\begin{proof}
The backward direction is \cite[Theorem 2.1(2)]{ga}. For the forward
direction assume that $R\bowtie^f J$ is Gaussian. Then $R$ is a
Gaussian ring, since the Gaussian property is stable under factor
rings. Next, let $i,j\in J$. Assume that $j=f(a)$ for some $a\in R$.
By assumption we have $((a,j), (0,j))^2=((0,j)^2)$ or $((a,j)^2)$
\cite[Theorem 2.2]{T}. If the first case happens, then $a^2=0$; so
that $j^2=0$. The second case yields $j^2(1-j')=0$ for some $j'\in
J$. Thus $j^2=0$. Similarly $i^2=0$. Hence $ij=0$. To this end one
notices that $R\bowtie^f J$ is Gaussian. Let now $0\neq r\in \fm$
and $j\in J$. The last equality is the direct consequence of the
statement $((r,f(r)), (0,j))^2=((r,f(r))^2)$ or $((0,j)^2)$ together
with $J^2=0$.
\end{proof}

In the course of this paper, for an  $R$-module $M$, $\Supp_R(M)$ will
denote the set of all prime ideals $\fp$ of $R$ such that
$M_\fp\neq0$.

In a sense, the next corollary is a generalization of \cite[Corollary
3.8(3)]{k}.

\begin{cor}\label{cg}
Assume that $J\subseteq f(R)\cap \Jac (S)$. Then $R\bowtie^f J$ is
Gaussian if and only if $R$ is Gaussian, $J_\fm^2=0$ and
$f(r)J_\fm=f(r)^2J_\fm$ for every $\fm\in\Max(R)\cap\V(f^{-1}(J))$
and $r\in \fm$.
\end{cor}
\begin{proof}
One can see $J_\fm= J_{T_\fm}$ for every $\fm\in\V(f^{-1}(J))$ since
$T_\fm=f(R\setminus\fm)+J=f(R\setminus\fm)$. It also follows from
the assumption $J\subseteq f(R)$ that $J_{T_\fm}\subseteq
f_\fm(R_\fm)$ for every $\fm\in\Supp_R(J)\cap\V(f^{-1}(J))$. Hence
Theorem \ref{g} completes the proof.
\end{proof}

The next example illustrates the assumption $J\subseteq f(R)$ in the
forward direction of the theorem above is essential.

\begin{exam}\label{ex}
Let $\mathbb{Z}$ be the ring of integers and $\fp=p\mathbb{Z}$ be
the prime ideal generated by a prime number $p$. Let $X$ be an
indeterminate over the field of rational numbers $\mathbb{Q}$, and
let $R:=\mathbb{Z}_{\fp}$, $S:=\mathbb{Q}\llbracket X \rrbracket$
the formal power series ring over $\mathbb{Q}$,
$J:=X\mathbb{Q}\llbracket X \rrbracket$ the unique maximal ideal of
$S$, $f:R\to S$ be the inclusion homomorphism. It is clear that
$J\nsubseteq f(R)$. Notice that $R\bowtie^f J$ is isomorphic to the
composite ring extension $\mathbb{Z}_\fp+X\mathbb{Q}\llbracket X
\rrbracket$, by \cite[Example 2.5]{DFF}. It follows from
\cite[Theorem 1.3]{ht07} that $R\bowtie^f J$ is a Pr\"{u}fer domain,
hence, a Gaussian ring \cite[Corollary 28.5]{g72}, while $J^2\neq0$.
\end{exam}

As an application of our results we construct some examples of Pr\"{u}fer rings
which are not Gaussian.

\begin{exam}\label{epng}
Let $k$ be a field and $X$ an indeterminate over $k$. Let
$R:=k[X]/(X^8)$ which is a total quotient ring, $S:=k[X]/(X^4)$,
$J:=(X^2)/(X^4)$ and $f$ be the canonical surjection. Then
$R\bowtie^f J$ has the condition $\star$ since $f$ is surjective. It
follows from Proposition \ref{T}(1) that $R\bowtie^f J$ is a total
quotient ring, hence Pr\"{u}fer. However $R\bowtie^f J$ is not a
Gaussian ring by Theorem \ref{g}, since $XJ\neq X^2J$.
\end{exam}

The next example handles the non-local case.

\begin{exam}
Let $R:=\mathbb{Z}/48\mathbb{Z}$, $S:=\mathbb{Z}/24\mathbb{Z}$,
$J:=6\mathbb{Z}/24\mathbb{Z}$ and $f$ be the canonical surjection.
Then, as Example \ref{epng}, $R\bowtie^f J$ has the condition $\star$ and it
follows from Proposition \ref{T}(1) that $R\bowtie^f J$ is a
Pr\"{u}fer ring. But, it is not a
Gaussian ring by Corollary \ref{cg}, since the natural
image of $12=6^2-24$ in $J_\fm^2$ is not zero, where $\fm=2\mathbb{Z}/48\mathbb{Z}$.
\end{exam}

\section{Transfer of arithmetical condition}

This section deals with arithmetical condition on amalgamated
construction. The ring $R$ is said to be an \emph{arithmetical} ring if every
finitely generated ideal of $R$ is locally principal \cite{j66}.
Thus $R$ is an arithmetical ring if and only if $R$ is a locally
chain ring. Recall that an $R$-module $M$ is said to be
\emph{uniserial} if its set of submodules is totally ordered by
inclusion and $R$ is a \emph{chain ring} if it is uniserial as
$R$-module.

The special case of the following result has been appeared in
\cite[Proposition 1.1]{c15}.

\begin{thm}\label{chain}
Assume that $J$ is a non-zero ideal of $S$. If $R\bowtie^f J$ is a
chain ring, then $R$ is a valuation domain and $J=(f(a)+j)J$ for
every $0\neq a\in R$ and $j\in J$. The converse holds provided that
$J$ is a uniserial $R$-module.
\end{thm}
\begin{proof}
Assume that $R\bowtie^f J$ is a chain ring. Then, as a factor ring,
$R$ is a chain ring. Let $0\neq a\in R$ and $l\in J$. Clearly
$(a,f(a))\notin(0,l)(R\bowtie^f J)$. Then $(0,l)=(r,f(r)+i)(a,f(a))$
for some $(r,f(r)+i)\in R\bowtie^f J$; so that $ra=0$ and $l=f(a)i\in
f(a)J$. Hence $J=f(a)J$ and this shows that $R$ is a valuation
domain. Now if further $j\in J$, there is $(r,f(r)+k)\in R\bowtie^f
J$ such that $(0,l)=(a,f(a)+j)(r,f(r)+k)$. Hence $ra=0$ which
implies that $r=0$. Thus $l=(f(a)+j)k\in (f(a)+j)J$. Therefore
$J=(f(a)+j)J$.

For the converse assume that $J$ is a uniserial $R$-module. Let
$(a,f(a)+j), (b,f(b)+i)\in R\bowtie^f J$ be arbitrary. Assume that
$b\neq0$ and $a=bc$. Then there exists $k\in J$ such that
$j-f(c)i=(f(a)+i)k$. Hence $(a,f(a)+j)=(b,f(b)+i)(c,f(c)+k)$. In the
case that $b=0$, there is $r\in R$ such that $j=ri(=f(r)i)$ since
$J$ is uniserial. Hence $(0,j)=(r,f(r))(0,i)$.
\end{proof}

We notice that one can not necessarily
deduce $J$ is a uniserial $R$-module from the assumption that
$R\bowtie^f J$ is a chain ring. In fact,
bearing in mind the notations of Example \ref{ex}, it is clear that the
$\mathbb{Z}_{\fp}$-submodules of $J$ generated by the
two elements $X$ and $X^2$ are not comparable.

\begin{cor}\label{J}
The following statements hold:
    \begin{enumerate}
        \item Assume that $J$ is a non-zero ideal of $S$ such that $J^2=0$.
        Then $R\bowtie^f J$ is a chain ring if and only
if $R$ is a valuation domain, $J$ is a uniserial $R$-module and
$J=f(a)J$ for every $0\neq a\in R$.
        \item Assume that $J\subseteq f(R)$. Then $R\bowtie^f J$ is a chain ring if and only
if $R$ is a chain ring and $J=0$.
    \end{enumerate}
\end{cor}
\begin{proof}
(1) It is obvious that $J=(f(a)+j)J$ for every
$0\neq a\in R$ and $j\in J$ if and only if $J=f(a)J$ for
every $0\neq a\in R$. Then, in view of Theorem \ref{chain},
it remains for us to deduce $J$ is
uniserial when $R\bowtie^f J$ is a chain ring. Indeed, for $i,j\in
J$, one of the elements $(0,i)$ and $(0,j)$ divides the other. So we
may assume $(0,j)=(r,f(r)+k)(0,i)$ for some $(r,f(r)+k)\in
R\bowtie^f J$. Thus $j=(f(r)+k)i=f(r)i$.

(2) One implication is clear. For the other, suppose,
on the contrary, that $J\neq 0$. Then one can choose
a non-zero element $j\in J$ and, hence, a non-zero element $a\in R$ such
that $j=f(a)$. This together with Theorem \ref{chain} shows that
$J=(f(a)-j)J=0$, which is a
contradiction.
\end{proof}

\begin{cor}(see \cite[Theorem 3.2(1)]{k})
The ring $R\bowtie I$ is a chain ring if and only if $R$ is a chain
ring and $I=0$.
\end{cor}

We now want to generalize Corollary \ref{J} to arithmetical rings.
To this end, we need the concept of distributive
lattice of submodules.
Let $M$ be a module over the ring $R$. We say that $M$ has a
\emph{distributive} lattice of submodules if $M$ satisfies one of
the following two equivalent conditions:
\begin{enumerate}
\item $(N+L)\cap K=(N\cap K)+(L\cap K)$ for every submodules $N,
  L, K$ of $M$;
\item $(N\cap L)+ K=(N+K)\cap(L+K)$ for every submodules $N,
  L, K$ of $M$.
\end{enumerate}
It is known that $M$ has a distributive lattice of submodules if and
only if $M_\fm$ is a uniserial module for each maximal ideal $\fm$
of $R$ \cite[Proposition 1.2]{c15}.

The special case of part 2 of the following result was obtained in
\cite[Corollary 3.8(1)]{k}.

\begin{cor}\label{v}
The following statements hold:
\begin{enumerate}
  \item Assume that $J$ is a non-zero ideal of $S$ with
  $J^2=0$. Then $R\bowtie^f J$ is an
arithmetical ring if and only if $R$ is arithmetical, $R_\fp$ is a
domain for every $\fp\in\Supp_R(J)$, $J$ is locally divisible and
has a distributive lattice of submodules.
  \item Assume that $J\subseteq f(R)$. Then $R\bowtie^f J$ is an
arithmetical ring if and only if $R$ is arithmetical,
$J_\fm=0$ for every $\fm\in\Max(R)\cap\V(f^{-1}(J))$
and $S_\fq$ is a chain ring for every $\fq\in\Max(S)\setminus\V(J)$.
\end{enumerate}
\end{cor}
\begin{proof}
(1) One can employ \cite[Lemma 2.8]{sss16}
together with Corollary \ref{J} and the above-mentioned
result \cite[Proposition 1.2]{c15} to deduce the assertion.

(2) Assume that $R\bowtie^f J$ is an arithmetical ring.
First, note that, as a factor ring,
$R$ is an arithmetical ring. Then we show that $J_\fm=0$
for every $\fm\in\Max(R)\cap\V(f^{-1}(J))$. Let $\fm\in\Max(R)\cap\V(f^{-1}(J))$.
By assumptions, $(R\bowtie^fJ)_{\fm^{\prime_f}}\cong R_{\fm}\bowtie^{f_{\fm}}J_{T_{\fm}}$
is a chain ring and one has the inclusion $J_{T_{\fm}}\subseteq f_{\fm}(R_{\fm})$.
Hence $J_{T_{\fm}}=0$ by Corollary \ref{J}. On the other hand, using
the inclusions $J\subseteq f(R)$ and $f^{-1}(J)\subseteq \fm$, it is easily seen that
$T_\fm=f(R\setminus\fm)+J=f(R\setminus\fm)$. Therefore $J_\fm=0$.
Finally, the isomorphism $(R \bowtie^{f}J)_{\overline{\fq}^f}\cong S_{\fq}$
shows that $S_\fq$ is a chain ring for every $\fq\in\Max(S)\setminus\V(J)$.
The converse direction is easy to obtain.
\end{proof}

The conclusion of the corollary above fails if the
assumption $J^2=0$ or $J\subseteq f(R)$ is dropped. For example,
let $X$ be an indeterminate over the field of rational numbers
$\mathbb{Q}$, and let $R:=\mathbb{Z}$, $S:=\mathbb{Q}\llbracket X
\rrbracket$ the formal power series ring over $\mathbb{Q}$,
$J:=X\mathbb{Q}\llbracket X \rrbracket$, $f:R\to S$ be the inclusion
homomorphism. It is clear that $J\nsubseteq f(R)$ and that
$J^2\neq0$. Notice that $R\bowtie^f J
\cong\mathbb{Z}+X\mathbb{Q}\llbracket X \rrbracket$ \cite[Example
2.5]{DFF}. It follows from \cite[Theorem 1.3]{ht07} that $R\bowtie^f
J$ is a Pr\"{u}fer domain, hence, an arithmetical ring. However $J$
is not locally divisible.

In concluding we give an example of a non-arithmetical Gaussian ring.

\begin{exam}
Let $k$ be a field and $X$ an indeterminate over $k$. Set $R:=k[X]$,
$S:=k[X]/(X^2)$, $J:=(X)/(X^2)$ and let $f$ be the canonical
surjection. Then $R\bowtie^f J$ is Gaussian by Corollary \ref{cg}.
However $R\bowtie^f J$ is not an arithmetical ring by Corollary
\ref{v}, since $J\neq 0=XJ$.
\end{exam}



\end{document}